\documentclass[english,5p,authoryear,twocolumn,fleqn]{elsarticle}

\usepackage[T1]{fontenc}
\usepackage[latin9]{inputenc}
\usepackage{natbib}        
\usepackage{ragged2e}
\usepackage{amsthm,amsmath,amssymb}
\usepackage{graphicx,xcolor}

\usepackage{array}
\usepackage{multirow}

\usepackage{babel}
\addto{\captionsenglish}{\renewcommand*{\appendixname}{Appendix }}

\graphicspath{{./Figs/}}

\journal{Automatica}

\newtheorem{thm}{Theorem}

\newtheorem{lem}{Lemma}

\begin{document}
\begin{frontmatter}

\title{\vspace{-24pt}Output Feedback Control Based on State and Disturbance Estimation}

\author[a]{Wuhua~Hu\corref{cor}}
\ead{w.hu920@gmail.com}

\author[b]{Eduardo~F.~Camacho}
\ead{efcamacho@us.es}

\author[c]{Lihua~Xie}
\ead{elhxie@ntu.edu.sg}

\cortext[cor]{Corresponding author}

\address[a]{Institute for Infocomm Research, Agency for Science,
Technology and Research (A{*}STAR), Singapore}

\address[b]{Department of System Engineering and Automatic Control,
Escuela Superior de Ingenieros of the University of Sevilla, Spain}

\address[c]{School of Electrical and Electronic Engineering, Nanyang Technological
University, Singapore\vspace{-12pt}}

\begin{abstract}
Recently developed control methods with strong disturbance rejection
capabilities provide a useful option for control design. The key lies
in a general concept of disturbance and effective ways to estimate
and compensate the disturbance. 
This work extends the concept of disturbance as the mismatch between a system model and the true dynamics, and estimates and compensates the disturbance for multi-input multi-output linear/nonlinear systems described in a general form. 
{
The results presented do not need to assume the disturbance to be independent of the control inputs or satisfy a certain matching condition, and do not require the system to be expressible in an integral canonical form as required by algorithms previously described in literature. 
}%
The estimator and controller are designed under a state tracking framework, and sufficient conditions for the closed-loop stability are presented.  
The performance of the resulting
controller relies on a co-design
of the system model, the state and disturbance observer, and the controller. Numerical experiments
on a first-order system and an inverted pendulum under uncertainties are used to illustrate
the control design method and demonstrate its efficacy. 
\end{abstract}

\begin{keyword}
Linear systems, nonlinear systems, output feedback, state and disturbance
estimation, disturbance compensation, reference tracking, inverted
pendulum
\end{keyword}

\end{frontmatter}

\section{Introduction} \label{sec: introduction}

Classic and modern control theories rely heavily on the model of the
system under control. The control performance and robustness are
largely determined by fidelity of the model. In practice, it often
requires huge amount of effort to develop
a proper model and then a controller. The model and controller need to compromise
between the design and operation complexity (or cost) and the achievable
performance and robustness. Depending on the chosen trade-offs, the model
mismatch relative to the true system dynamics can be large or small,
and its effects are often tolerated by robust
control \citep{Zhou1998} or mitigated by adaptive control
\citep{aastrom2013adaptive}. The adopted controller also has to balance
various trade-offs involved. So far, proportional-integral-derivative
(PID) control seems to be the only practical option that well balances the trade-offs,
explaining the fact that PID control dominates more than 90\% of the
control application market \citep{O'Dwyer2009,Kano2010}. The situation
is unlikely to be changed until a more competitive method is developed
which can balance common trade-offs in various applications in a more
cost-effective manner.

Seeking for such an alternative method motivates the recent developments
of disturbance rejection based methods for control. These methods
emphasize the central task of control to reject disturbances \citep{johnson1986disturbance,gao2014centrality},
which is consistent with the original motivation of introducing feedback
control \citep{truxall1955automatic}. The new mind breakthrough lies
in a generalized concept of disturbance and effective ways to estimate
and compensate the disturbance. The ideas and related developments
scatter in literature, leading to several useful control methods such
as disturbance accommodating control (DAC) \citep{johnson1968optimal,johnson1970further,johnson1986disturbance},
disturbance observer based control (DOBC) \citep{ohishi1983torque,ohishi1987microprocessor,chen2016disturbance},
active disturbance rejection control (ADRC) \citep{Han1998,Han1999,Gao2006},
uncertainty and disturbance estimator (UDE) based control \citep{Youcef-Toumi1988,Zhong2004,Zhong2011},
model-free control (MFC) \citep{Fliess2009,Fliess2013}, etc. These
methods emerge almost independently, and their close relations begin
to be noticed recently \citep{schrijver2002disturbance,gao2014centrality,chen2016disturbance}.

DAC was introduced by Johnson \citeyearpar{johnson1968optimal,johnson1976theory,johnson1986disturbance}.
The key idea is to treat disturbances as additional states and to estimate them together with the system states using a composite state
observer. The disturbance estimates are used to counteract the
actual disturbances. The closed-loop stability with DAC was proved for systems described in linear forms in which the disturbances may be
explicitly dependent on the system states but not the control inputs. 

DOBC was proposed by Ohnishi and his colleagues \citep{ohishi1983torque,ohishi1987microprocessor}.
It is mainly developed for minimum phase systems because it involves
the inverse of a nominal plant. It may also be modified for non-minimum
phase systems \citep{shim2008new} but with a degraded disturbance
rejection ability \citep{shim2009almost}. Conventional DOBC is limited
for its frequency-domain design which requires disturbances to satisfy
the so-called matching conditions, namely, the disturbances enter
a system via the same channels as control inputs, or can be transformed
into inputs in the same channels as the control inputs by change of coordinates
\citep{schrijver2002disturbance}. The closed-loop stability of a
single-input single-output (SISO) system with DOBC was analyzed in
\citep{shim2009almost}, and some guidelines to tune the observers
were provided in \citep{schrijver2002disturbance,sariyildiz2014guide}. Further development which combines DOBC and existing control methods (e.g., robust control) to handle systems with multiple types of disturbances is also available in the literature. Interested readers are referred to the references \citep{guo2005disturbance,yao2014disturbance} for some relevant studies, and \citep{guo2014anti} for a brief survey.

ADRC was introduced by Han \citeyearpar{Han1995,Han1998,Han1999}.
It relies on the \textit{observation} that many dynamical systems
can be transformed into an integral canonical form as represented by a cascade
of integrators via certain (normally unknown) input-dependent state transformations \citep{Han1981}.
By extracting a nominal model from the canonical form, unmodeled
dynamics are lumped as a\textit{ total disturbance}, which is then
estimated and compensated online. The canonical form is similar to
the one proved by Fliess \citeyearpar{Fliess1990} for general dynamical
systems, and can be viewed as a special form of the canonical form
presented in \citep{Youcef-Toumi1988}. The idea of lumping disturbances
and uncertainties also coincides with that of DAC. In ADRC, the total
disturbance is treated as an additional state and then estimated jointly
with the system states via the so-called extended state observer (ESO)
\citep{Han1995}. The observer is essentially the same as the composite
state observer used in DAC \citep{meditch1974observers,johnson1975observers}.
Despite the overlap of ideas, ADRC does bring a merit for its unique
interest in designing a controller based on the integral canonical form. The theoretical basis
of ADRC is still lacking, however. Related analyses have been mainly on the capacity
of the disturbance observer \citep{Yang2009,guo2011convergence,Zheng2012,Huang2012}
and the closed-loop stability for SISO systems in the integral canonical
form \citep{Zheng2007a}. Some stability results are also available
for certain classes of multi-input multi-output (MIMO) systems \citep{Huang2012,guo2013convergence},
but the generalization to general MIMO systems is yet unclear.

Another idea for rejecting the total disturbance was suggested by
Youcef-Toumi and Ito \citeyearpar{Youcef-Toumi1988}. The authors
proposed the so-called time delay control (TDC) to make a nonlinear
system with uncertainties track reference dynamics. TDC uses past
observation of the disturbance to approximately cancel the current
one. The closed-loop performance is then governed by state feedback
and model reference feedforward controls. Later, Zhong and Rees generalized
TDC by replacing the time-delay filter with a general low pass filter,
resulting in the UDE-based control \citep{Zhong2004,Zhong2011}. To
date, the stability of UDE-based control has been proved for linear
time-invariant (LTI) SISO systems with first-order disturbance filters,
and its application relies on the assumption that all system states
are available.

Another closely related method, MFC, was introduced by Fliess and
Join \citeyearpar{Fliess2009,Fliess2013}. The method approximates
a continuous-time system by a local model within a very short time
period. The model mismatch is estimated and canceled online using
an algebraic identification technique which was proposed in \citep{fliess2004algebraic,fliess2008closed}
and later improved in \citep{Hu2014improved}. Consequently, the local model reduces to a cascade of integrators
for which feedback control design becomes straightforward. By specifying
a first- or second-order local model, MFC enables PID feedback control
to be embedded with direct disturbance rejection capability for output
tracking, yielding the so-called intelligent PID control \citep{Fliess2009,Fliess2013}.
To date, MFC has been studied mainly for SISO systems though its extension
to MIMO systems is thought to be possible \citep{Fliess2013}. Furthermore,
a rigorous stability analysis of MFC is still missing.

Other disturbance observer based control methods can be referred to
a recent survey made in \citep{chen2016disturbance}. All of these
methods are essentially different manifestations of the same philosophy
of feedback control with explicit disturbance estimation and compensation.
The methods differ mainly in how disturbances are defined and how
they are estimated and compensated. Generally speaking, the disturbances
are treated as the mismatch between a nominal model and the true system. The
disturbances can be estimated by a composite/extended state observer
in time domain as used by DAC and ADRC for general cases, or by a
disturbance observer used by DOBC, UDE-based control or MFC for more
restricted cases.

The above methods mostly assumed that the unmodelled dynamics are independent of the control inputs, and concerned with SISO or restricted MIMO systems
with limited stability analyses. Motivated by the indicated limitations, this
work is devoted to developing a new rigorous framework for control of
general linear/nonlinear MIMO systems via output feedback which founds
on the philosophy of disturbance estimation and compensation, and meanwhile adopts a most general interpretation of the disturbance. 
{
In contrast to other methods described in literature, the MIMO system does not need to be in (or transformable into) the form of cascaded integrators, and furthermore, the lumped disturbance includes unmodelled dynamics which can be dependent on the control inputs.
}%
The closed-loop stability is analyzed with
minimum assumptions on the system, model and controller. In addition to the more comprehensive literature review presented above, this work
extends our conference paper presented in the 19th IFAC world congress
\citep{hu2014feedforward} at several aspects. Firstly, the output
feedback design is presented under a more general setting in which
partial knowledge about the system can be incorporated into the system
model. Secondly, the stability analysis of the proposed control is refined and improved. Thirdly, the numerical studies are enhanced for better illustration and validation of the theoretical results. 

\textit{Notation. }Scalars are denoted by normal small/capital letters,
vectors by bold small letters, and matrices by bold capital letters.
The symbol $\mathbf{0}_{n}$ stands for an $n\times1$ full zero vector,
and $\mathbf{O}_{n\times m}$ an $n\times m$ full zero matrix. $\mathbf{O}_{n}$
refers to $\mathbf{O}_{n\times n}$, and $\mathbf{I}_{n}$ represents
an $n\times n$ identity matrix. Matrices $\mathbf{X}^{\top}$ and $\mathbf{X}^{\dagger}$ denote the
transpose and Moore-Penrose pseudoinverse of the matrix $\mathbf{X}$, respectively, and $\sigma_{\min}(\mathbf{X})$
and $\sigma_{\max}(\mathbf{X})$ represent its minimum and maximum
singular values, respectively. $\mathbf{X}\succ \mathbf{O}_n$ means that $\mathbf{X}$ is an $n\times n$ positive definite matrix. $\left\Vert \bullet\right\Vert $ refers
to the 2-norm of a matrix or vector. The arguments of a variable or
function are ignored whenever no ambiguity arises.

\section{Problem formulation \label{sec:Problem-formulation}}

Consider a dynamical system described by 
\begin{equation}
\begin{aligned}\dot{\mathbf{x}}(t) & =\mathbf{f}_{0}(t,\,\mathbf{x},\,\mathbf{u},\,\mathbf{w}_{0}),\\
\mathbf{y}(t) & =\mathbf{g}_{0}(t,\,\mathbf{x},\,\mathbf{u},\,\mathbf{v}_{0}),
\end{aligned}
\label{eq: true system}
\end{equation}
for $t\ge t_{0}$, where $t\in\mathbb{R}$ is the time, and $\mathbf{x}\in\mathbb{R}^{n}$,
$\mathbf{u}\in\mathbb{R}^{m}$, $\mathbf{w}_{0}\in\mathbb{R}^{k_{0}}$,
$\mathbf{y}\in\mathbb{R}^{l}$ and $\mathbf{v}_{0}\in\mathbb{R}^{p_{0}}$,
are vectors of the state, control input, system disturbance, measured
output, and measurement noise, respectively. The dimensions of the
variables satisfy $l\le m\le n$. Since the true state and measurement
functions, $\mathbf{f}_{0}$ and $\mathbf{g}_{0}$, are not exactly
known in practice, we consider a model of the system instead: 
\begin{equation}
\begin{aligned}\dot{\mathbf{x}}(t) & =\mathbf{f}(t,\,\mathbf{x},\,\mathbf{u})+\Gamma\mathbf{w}(t,\,\mathbf{x},\,\mathbf{u},\,\mathbf{w}_{0}),\\
\mathbf{y}(t) & =\mathbf{g}(t,\,\mathbf{x})+\Pi\mathbf{v}(t,\,\mathbf{x},\,\mathbf{u},\,\mathbf{v}_{0}),
\end{aligned}
\label{eq: general model}
\end{equation}
where $\mathbf{w}\in\mathbb{R}^{k}$ ($k\le n$) and $\mathbf{v}\in\mathbb{R}^{p}$
($p\le l$) lump all unmodeled dynamics in the state and measurement
functions, respectively, and $\Gamma\in\mathbb{R}^{n\times k}$ and
$\Pi\in\mathbb{R}^{l\times p}$ are tall matrices with zero
and non-zero elements indicating the sources of the mismatches
for each state and measurement, respectively. In the worst case when $\mathbf{f}_{0}$
is completely unknown and $\mathbf{f}$ is artificially assigned,
we will have $k=n$ and $\Gamma=\mathbf{I}_{n}$. Likewise, when $\mathbf{g}_{0}$
is completely unknown and $\mathbf{g}$ is artificial, we will have
$p=l$ and $\Pi=\mathbf{I}_{l}$. We highlight that the lumped disturbances $\mathbf{w}$ and $\mathbf{v}$ may both be dependent on the control input $\mathbf{u}$, which makes it challenging to design a stabilizing controller and is in sharp contrast to most literature reviewed in the previous section.

Based on the system model \eqref{eq: general model}, the control
problem is stated as follows. The system needs to track a reference
state trajectory, generated by 
\begin{equation}
\dot{\mathbf{x}}_{r}(t)=\mathbf{f}_{r}(t,\,\mathbf{x}_{r},\,\mathbf{u}_{r}),\label{eq: reference state dynamics}
\end{equation}
for $t\ge t_{0}$, where $\mathbf{x}_{r}\in\mathbb{R}^{n}$ is the
reference state trajectory and $\mathbf{u}_{r}\in\mathbb{R}^{m'}$
is the input used to excite the reference system (note that the dimension
of $\mathbf{u}_{r}$ can be different from that of $\mathbf{u}$).
To meet design specifications, the desired tracking error dynamics
is imposed as 
\begin{equation}
\dot{\mathbf{e}}=\mathbf{h}(t,\,\mathbf{e}),\label{eq: target state error dynamics}
\end{equation}
where $\mathbf{e}:=\mathbf{x}_{r}-\mathbf{x}$ which defines the error.
With \eqref{eq: general model} and \eqref{eq: reference state dynamics},
it follows that $\mathbf{f}_{r}(t,\,\mathbf{x}_{r},\,\mathbf{u}_{r})-\mathbf{f}(t,\,\mathbf{x},\,\mathbf{u})-\Gamma\mathbf{w}=\mathbf{h}(t,\,\mathbf{e})$,
from which an \textit{ideal} or \textit{desired} control $\mathbf{u}$ is solved. Since the true
state $\mathbf{x}$ and disturbance $\mathbf{w}$ are unavailable
in practice, they have to be estimated from the measurement $\mathbf{y}$.
Let their estimates be obtained as $\hat{\mathbf{x}}$ and $\hat{\mathbf{w}}$,
respectively, and define $\hat{\mathbf{e}}=\mathbf{x}_{r}-\hat{\mathbf{x}}$.
Then the equation becomes 
\begin{equation}
\mathbf{f}_{r}(t,\,\mathbf{x}_{r},\,\mathbf{u}_{r})-\Gamma\hat{\mathbf{w}}-\mathbf{f}(t,\,\hat{\mathbf{x}},\,\mathbf{u})-\mathbf{h}(t,\,\hat{\mathbf{e}})=\delta_{\mathbf{f}}+\Gamma\delta_{\mathbf{w}}+\delta_{\mathbf{h}},\label{eq: Key eq}
\end{equation}
where $\delta_{\mathbf{f}}:=\mathbf{f}(t,\,\mathbf{x},\,\mathbf{u})-\mathbf{f}(t,\,\hat{\mathbf{x}},\,\mathbf{u})$,
$\delta_{\mathbf{w}}:=\mathbf{w}-\hat{\mathbf{w}}$ and $\delta_{\mathbf{h}}:=\mathbf{h}(t,\,\mathbf{e})-\mathbf{h}(t,\,\hat{\mathbf{e}})$,
which are errors caused by imperfect estimation. As a consequence,
the desired control vector $\mathbf{u}$ has to be estimated from
the design equation \eqref{eq: Key eq}.

The system model in use affects the estimation and also the derived
control. Depending on the information available and the intention
for an affordable design, two cases can be considered: i) \textit{a
state or measurement model is available, and} ii)
\textit{neither a state nor a measurement model is available}. Case
ii) occurs when modeling of the system dynamics is difficult or costly, and can be viewed as an extreme case of i) when the available model is useless. Next
we present a feedback design to tackle Case i).

\section{Feedback control design \label{sec:Feedback-control-design}}

The feedback controller is developed, followed by two types of state
and disturbance estimators.

\subsection{The feedback controller form}

Given a system model in the form of \eqref{eq: general model}, the
desired control $\mathbf{u}$ is estimated from (\ref{eq: Key eq})
subject to (\ref{eq: general model}). The estimation, however, may
be intractable if the model functions $\mathbf{f}$ and $\mathbf{g}$
are nonlinear. In that case, we further approximate the system using
a linear model while exploiting the available information as much
as possible. Hereafter, we assume that $\mathbf{f}$ and $\mathbf{g}$
are linear functions and that $\mathbf{w}$ and $\mathbf{v}$ lump
all model mismatches. With such a system model, it is possible to
derive an estimate of $\mathbf{u}$ in a closed form.

Let the state and measurement functions be given as 
\begin{equation}
\begin{aligned}\mathbf{f}(t,\,\mathbf{x},\,\mathbf{u}) & :=\mathbf{A}\mathbf{x}+\mathbf{B}\mathbf{u},\\
\mathbf{g}(t,\,\mathbf{x}) & :=\mathbf{C}\mathbf{x}+\mathbf{Du},
\end{aligned}
\label{eq: fictitious_system_model}
\end{equation}
where $\mathbf{A}$, $\mathbf{B}$, $\mathbf{C}$ and $\mathbf{D}$
are constant matrices of compatible dimensions, with $(\mathbf{A},\,\mathbf{B})$
being controllable and $(\mathbf{A},\,\mathbf{C})$ being observable.
Then the key equation \eqref{eq: Key eq} becomes{\small{}{} 
\begin{equation}
\mathbf{f}_{r}(t,\,\mathbf{x}_{r},\,\mathbf{u}_{r})-\mathbf{A}\hat{\mathbf{x}}-\Gamma\hat{\mathbf{w}}-\mathbf{h}(t,\,\hat{\mathbf{e}})-\mathbf{B}\mathbf{u}=\mathbf{A}\mathbf{\delta}_{\mathbf{x}}+\Gamma\mathbf{\delta}_{\mathbf{w}}+\delta_{\mathbf{h}}.\label{eq: exemplary-key-equation}
\end{equation}
}The equation embodies a feedforward signal, $\mathbf{f}_{r}(t,\,\mathbf{x}_{r},\,\mathbf{u}_{r})$,
which embeds a reference state trajectory, and two feedback signals,
$\mathbf{A}\hat{\mathbf{x}}$ and $\Gamma\hat{\mathbf{w}}$, which
try to cancel the modeled state dynamics, and another feedback signal
$\mathbf{h}(t,\,\hat{\mathbf{e}})$, which tries to enforce the desired
tracking error dynamics. The control estimate will be coded by all
of these signals.

Depending on how $\hat{\mathbf{x}}$ and $\hat{\mathbf{w}}$ are obtained,
different control estimates can be derived from \eqref{eq: exemplary-key-equation}.
If $\hat{\mathbf{x}}$ or $\hat{\mathbf{w}}$ has an explicit relation
to $\mathbf{u}$, then it is preferred to substitute the relation
into \eqref{eq: exemplary-key-equation} before estimating $\mathbf{u}$:
this will make the estimation of $\hat{\mathbf{x}}$ or $\hat{\mathbf{w}}$
transparent to the implementation, that is, the explicit estimate
will be avoided. 
{
Otherwise, it is straightforward to derive a least-square (LS) estimate of the desired control $\mathbf{u}$
as 
\begin{equation}
\hat{\mathbf{u}}=\mathbf{B}^{\dagger}\left(\mathbf{f}_{r}(t,\,\mathbf{x}_{r},\,\mathbf{u}_{r})-\Gamma\hat{\mathbf{w}}-\mathbf{A}\hat{\mathbf{x}}-\mathbf{h}(t,\,\hat{\mathbf{e}})\right),\label{eq: LS estimate of u}
\end{equation}
which minimizes the 2-norm of the left-hand-side of \eqref{eq: exemplary-key-equation}
(where $\mathbf{B}^{\dagger}:=(\mathbf{B}^{\top}\mathbf{B})^{-1}\mathbf{B}^{\top}$).
Note that $\hat{\mathbf{u}}$ is in general not equal to $\mathbf{u}$ because of the existence of observation errors. 
}%
This control estimate may introduce a bias to the equation even if
$\hat{\mathbf{x}}$ and $\hat{\mathbf{w}}$ are equal to the true
values. This is seen by replacing $\mathbf{u}$ in \eqref{eq: exemplary-key-equation}
with $\hat{\mathbf{u}}$, resulting in a bias {\vspace{-10pt}

\textit{\small{}{} 
\begin{equation}
\mathbf{\delta}_{\mathbf{u}}:=(\mathbf{I}_{n}-\mathbf{B}\mathbf{B}^{\dagger})\left(\mathbf{f}_{r}(t,\,\mathbf{x}_{r},\,\mathbf{u}_{r})-\Gamma\hat{\mathbf{w}}-\mathbf{A}\hat{\mathbf{x}}-\mathbf{h}(t,\,\hat{\mathbf{e}})\right).\label{eq: delta_u}
\end{equation}
}{\small{}}}The bias will be compensated by the next update of control
once it becomes part of the renewed disturbance $\mathbf{w}$. It
can be shown that the bias will be zero if the system model is in
the canonical forms introduced in \citep{Han1981,Youcef-Toumi1988,Fliess1990}.
However, it is worthwhile to mention that a zero bias does not imply
closed-loop stability of the control system, which will be clear from
the stability conditions presented in Section \ref{sec: Closed-loop-stability}.

\subsection{State and disturbance estimators}

Given the state and measurement functions in (\ref{eq: fictitious_system_model}),
the complete system model has the form of 
\begin{equation}
\begin{aligned}\dot{\mathbf{x}} & =\mathbf{A}\mathbf{x}+\mathbf{B}\hat{\mathbf{u}}+\Gamma\mathbf{w}(t,\,\mathbf{x},\,\hat{\mathbf{u}},\,\mathbf{w}_{0}),\\
\mathbf{y} & =\mathbf{C}\mathbf{x}+\mathbf{D}\hat{\mathbf{u}}+\Pi\mathbf{v}(t,\,\mathbf{x},\,\hat{\mathbf{u}},\,\mathbf{v}_{0}),
\end{aligned}
\label{eq: linear-fictitious-system-model}
\end{equation}
which is an LTI system subject to generalized disturbances and measurement
noises. To compute the control $\hat{\mathbf{u}}$ from (\ref{eq: exemplary-key-equation})
or implement it per (\ref{eq: LS estimate of u}), the state $\mathbf{x}$
and the disturbance $\mathbf{w}$ need to be estimated (implicitly
or explicitly) from the measurement $\mathbf{y}$. Depending on the
invertibility of matrix $\mathbf{C}$, two types of estimators can
be used for this purpose. If $\mathbf{C}$ is square and invertible,
then it is feasible to estimate $\mathbf{x}$ by a direct filtering
of $\mathbf{y}$ and further estimate $\mathbf{w}$ by use of another
filtering. This gives the Type-I estimator. Otherwise, an ESO can
be used to estimate $\mathbf{x}$ and $\mathbf{w}$ simultaneously,
resulting in the Type-II estimator. 
{
Type-II estimator is also applicable
when $\mathbf{C}$ is invertible, though its design may be more involved
in time domain compared to the design of Type-I estimator in frequency
domain. In this case, Type-I estimator has another advantage that
the disturbance need not be explicitly estimated.
}%

\textit{Type-I estimator (if $\mathbf{C}$ is invertible).} The state
$\mathbf{x}$ is estimated by directly filtering $\mathbf{y}$ as
$\begin{aligned}\hat{\mathbf{x}}(t) & =\mathbf{C}^{-1}\left(\mathbf{f}_{\mathbf{y}}(t).\star\left(\mathbf{y}(t)-\mathbf{D}\hat{\mathbf{u}}(t)\right)\right)\end{aligned}
$, where $\mathbf{f}_{\mathbf{y}}(t)\in\mathbb{R}^{l}$ is a vector
of impulse responses of filters that suppress the measurement noises
effectively, and $.\star$ denotes the element-wise convolution operator
which operates on two corresponding elements of the two vectors.

With the state estimate, the combined disturbance $\Gamma\mathbf{w}$
can then be estimated by applying filtering on the state equation,
yielding $\Gamma\hat{\mathbf{w}}(t)=\mathbf{f}_{\mathbf{\Gamma w}}(t).\star\left(\dot{\hat{\mathbf{x}}}(t)-\mathbf{A}\hat{\mathbf{x}}(t)-\mathbf{B}\hat{\mathbf{u}}(t)\right)$,
where $\mathbf{f}_{\mathbf{\Gamma w}}(t)\in\mathbb{R}^{n}$ is a vector
of impulse responses of proper filters, satisfying that $\mathbf{f}_{\mathbf{\Gamma w}}(t).\star\dot{\hat{\mathbf{x}}}(t)$
is realizable. (In general, different filters can be applied to the
three terms in the bracket.) Substitute $\Gamma\hat{\mathbf{w}}$
into the key equation (\ref{eq: exemplary-key-equation}), from which
the LS estimate of the desired control is obtained as {

\vspace{-12pt}
{\small{}
\begin{equation}
\hat{\mathbf{u}}(t)=-\mathbf{B}^{\dagger}\left(\mathbf{A}\hat{\mathbf{x}}(t)+\mathcal{L}^{-1}\left(\begin{array}{c}
s\mathbf{J}^{-1}(s)\mathbf{F}_{\mathbf{\Gamma w}}(s)\hat{\mathbf{X}}(s)\\
-s\mathbf{J}^{-1}(s)\mathbf{X}_{r}(s)\\
+\mathbf{J}^{-1}(s)\mathbf{H}_{\hat{\mathbf{e}}}(s)
\end{array}\right)\right),\label{eq: Type I-control-estimate}
\end{equation}
}}where $\mathbf{J}(s):=\mathbf{I}_{n}-\mathbf{F}_{\mathbf{\Gamma w}}(s)$,
and $\mathbf{F}_{\Gamma\mathbf{w}}(s)$ is a diagonal matrix in Laplace
domain whose diagonal elements correspond to the Laplacian transforms
of vector $\mathbf{f}_{\Gamma\mathbf{w}}(t)$; and $\hat{\mathbf{X}}(s)$,
$\mathbf{X}_{r}(s)$ and $\mathbf{H}_{\hat{\mathbf{e}}}(s)$ are the
Laplacian transforms of $\hat{\mathbf{x}}(t)$, $\mathbf{x}_{r}(t)$
and $\mathbf{h}(t,\thinspace\hat{\mathbf{e}})$, respectively; and
$\mathcal{L}^{-1}$ denotes the inverse Laplacian transform. The control
input in (\ref{eq: Type I-control-estimate}) is ready to be implemented,
without the need of explicitly computing the disturbance estimate.

The state and disturbance estimation errors depend on the filters
$\mathbf{f}_{\mathbf{y}}$ and $\mathbf{f}_{\Gamma\mathbf{w}}$ applied.
Designing $\mathbf{f}_{\mathbf{y}}$ properly needs to have prior
knowledge about the noise $\mathbf{v}$ which is usually accessible
in applications where it is independent of the applied control $\mathbf{u}$.
In contrast, designing $\mathbf{f}_{\Gamma\mathbf{w}}$ properly is
non-trivial and deserves particular investigations. This is because
the total disturbance $\Gamma\mathbf{w}$ contains all unmodeled dynamics
and may be dependent on the control input, which makes it difficult
to figure out an appropriate bandwidth for the filter design. Some
related results for linear time-varying systems can be found in \citep{Zhong2011}.

\textit{Type-II estimator (when $\mathbf{C}$ is either invertible
or not).} An ESO is used to estimate the state $\mathbf{x}$ and the
disturbance $\mathbf{w}$, simultaneously. Because the ESO has the
form of a conventional linear state observer, a Type-II estimator
is rather standard except for treating the total disturbance as an
extra state \citep{Han1995}.

Extend the state vector as $\bar{\mathbf{x}}=[\mathbf{x}^{\top}\,\,\mathbf{w}^{\top}]^{\top}$,
which is an $(n+k)\times1$ vector. Then the model equations in (\ref{eq: linear-fictitious-system-model})
can be rewritten as 
\begin{equation}
\begin{aligned}\dot{\bar{\mathbf{x}}} & =\bar{\mathbf{A}}\bar{\mathbf{x}}+\bar{\mathbf{B}}\hat{\mathbf{u}}+\mathbf{E}\tilde{\mathbf{w}},\\
\mathbf{y} & =\bar{\mathbf{C}}\bar{\mathbf{x}}+\mathbf{D}\hat{\mathbf{u}}+\Pi\mathbf{v},
\end{aligned}
\label{eq: fictitious-LTI-model-with-extended-state}
\end{equation}
where $\bar{\mathbf{A}}:=\begin{smallmatrix}\left[\begin{array}{cc}
\mathbf{A} & \Gamma\\
\mathbf{O}_{k\times n} & \mathbf{O}_{k}
\end{array}\right]\end{smallmatrix}$, $\bar{\mathbf{B}}:=\begin{smallmatrix}\left[\begin{array}{c}
\mathbf{B}\\
\mathbf{O}_{k\times m}
\end{array}\right]\end{smallmatrix}$, $\mathbf{E}:=\begin{smallmatrix}\left[\begin{array}{c}
\mathbf{O}_{n\times k}\\
\mathbf{I}_{k}
\end{array}\right]\end{smallmatrix},$ $\bar{\mathbf{C}}:=\begin{smallmatrix}\left[\begin{array}{cc}
\mathbf{C} & \mathbf{O}_{l\times k}\end{array}\right]\end{smallmatrix}$, and $\tilde{\mathbf{w}}:=\dot{\mathbf{w}}$. The derivative $\dot{\mathbf{w}}$
now acts as the uncertainty instead of $\mathbf{w}$. A linear observer
can be used to estimate $\bar{\mathbf{x}}$, as follows: 
\begin{equation}
\begin{aligned}\dot{\hat{\bar{\mathbf{x}}}} & =\bar{\mathbf{A}}\hat{\bar{\mathbf{x}}}+\bar{\mathbf{B}}\hat{\mathbf{u}}+\mathbf{L}(\mathbf{y}-\hat{\mathbf{y}}),\\
\hat{\mathbf{y}} & =\bar{\mathbf{C}}\hat{\bar{\mathbf{x}}}+\mathbf{D}\hat{\mathbf{u}},
\end{aligned}
\label{eq: ESO}
\end{equation}
where $\mathbf{L}\in\mathbb{R}^{(n+k)\times l}$ is selected such
that $\bar{\mathbf{A}}-\mathbf{L}\bar{\mathbf{C}}$ is Hurwitz. The
existence of such a matrix requires observability of the extended
system model in (\ref{eq: fictitious-LTI-model-with-extended-state}). 

Define $\tilde{\mathbf{A}}=\bar{\mathbf{A}}-\mathbf{L}\bar{\mathbf{C}}$,
and $\mathbf{\mathbf{\delta}}_{\bar{\mathbf{x}}}=\bar{\mathbf{x}}-\hat{\bar{\mathbf{x}}}$.
Then, the estimation error $\mathbf{\mathbf{\delta}}_{\bar{\mathbf{x}}}$
is bounded if $\dot{\mathbf{w}}$ and $\mathcal{\mathbf{v}}$ are
bounded. 
\begin{lem}
\label{lm: state observation error bound}\textup{(Bounded estimation
error under bounded disturbances)} If $\left\Vert \dot{\mathbf{w}}\right\Vert \le c_{\dot{\mathbf{w}}}$,
\textup{$\left\Vert \mathcal{\mathbf{v}}\right\Vert \le c_{\mathbf{v}}$}
and $\tilde{\mathbf{A}}$ is Hurwitz, then given an $(n+k)\times(n+k)$
positive definite matrix $\mathbf{Q}$, there exists a finite time
$T_{1}$ ($\ge t_{0}$) such that $\left\Vert \mathbf{\mathbf{\mathbf{\delta}}_{\bar{\mathbf{x}}}}\right\Vert \le\frac{\sigma_{\max}(\mathbf{P})}{\sigma_{\min}(\mathbf{Q})}(c_{\dot{\mathbf{w}}}+c_{\mathbf{v}}\left\Vert \mathbf{L}\Pi\right\Vert )$
for all $t\ge T_{1}$, where $\mathbf{P}$ is a solution to the Lyapunov
equation $\tilde{\mathbf{A}}^{\top}\mathbf{P}+\mathbf{P}\tilde{\mathbf{A}}=-2\mathbf{Q}$. 
\end{lem}
\begin{proof}
The estimation error evolves by $\dot{\mathbf{\mathbf{\delta}}}_{\bar{\mathbf{x}}}=\tilde{\mathbf{A}}\mathbf{\mathbf{\delta}}_{\bar{\mathbf{x}}}+\mathbf{d}$,
where $\mathbf{d}:=[\mathbf{0}_{n}^{\top}\,\,\dot{\mathbf{w}}^{\top}]^{\top}-\mathbf{L}\Pi\mathbf{v}$.
The rest of the proof is completed by applying the results of robust
stability of LTI systems in time domain \citep{patel1980quantitative}. 
\end{proof}

\section{Closed-loop stability \label{sec: Closed-loop-stability}}

The closed-loop stability is analyzed by means of the state tracking
error dynamics. With the system dynamics described in (\ref{eq: linear-fictitious-system-model})
and the control given in (\ref{eq: exemplary-key-equation}), the
tracking error dynamics can be deduced as follows:{

\vspace{-10pt}
{\footnotesize{}
\begin{equation}
\begin{aligned}\dot{\mathbf{e}} & =\dot{\mathbf{x}}_{r}-\dot{\mathbf{x}}=\dot{\mathbf{x}}_{r}-\Gamma\hat{\mathbf{w}}-\mathbf{A}\hat{\mathbf{x}}-\mathbf{B}\hat{\mathbf{u}}-\mathbf{A}\mathbf{\mathbf{\delta}}_{\mathbf{x}}-\Gamma\mathbf{\delta}_{\mathbf{w}}\\
 & =\mathbf{h}(t,\,\hat{\mathbf{e}})+\left(\begin{array}{c}
\mathbf{f}_{r}-\Gamma\hat{\mathbf{w}}-\mathbf{A}\hat{\mathbf{x}}-\mathbf{h}(t,\,\hat{\mathbf{e}})\\
-\mathbf{B}\mathbf{B}^{\dagger}\left(\mathbf{f}_{r}-\Gamma\hat{\mathbf{w}}-\mathbf{A}\hat{\mathbf{x}}-\mathbf{h}(t,\,\hat{\mathbf{e}})\right)-\mathbf{A}\mathbf{\mathbf{\delta}}_{\mathbf{x}}-\Gamma\mathbf{\delta}_{\mathbf{w}}
\end{array}\right)\\
 & =\mathbf{h}(t,\,\hat{\mathbf{e}})+(\mathbf{I}_{n}-\mathbf{B}\mathbf{B}^{\dagger})\left(\mathbf{f}_{r}-\Gamma\hat{\mathbf{w}}-\mathbf{A}\hat{\mathbf{x}}-\mathbf{h}(t,\,\hat{\mathbf{e}})\right)-\mathbf{A}\mathbf{\mathbf{\delta}}_{\mathbf{x}}-\Gamma\mathbf{\delta}_{\mathbf{w}},\\
 & =\mathbf{h}(t,\,\mathbf{e})+\mathbf{\delta}_{\mathbf{u}}-\mathbf{A}\mathbf{\mathbf{\delta}}_{\mathbf{x}}-\Gamma\mathbf{\delta}_{\mathbf{w}}-\mathbf{\delta}_{\mathbf{h}}=\mathbf{h}(t,\,\mathbf{e})+\xi,
\end{aligned}
\label{eq: final-tracking-error-dynamics}
\end{equation}
}}where $\xi:=\mathbf{\delta}_{\mathbf{u}}-\mathbf{A}\mathbf{\mathbf{\delta}}_{\mathbf{x}}-\Gamma\mathbf{\delta}_{\mathbf{w}}-\mathbf{\delta}_{\mathbf{h}}$,
which defines the total design error that lumps the control-estimate-induced
bias and the errors caused by inexact estimation of $\mathbf{x}$
and $\mathbf{w}$. Sufficient conditions for the tracking error to
be bounded are given in the next theorem. 
\begin{thm}
\textup{\label{thm: Weak-closed-loop-stability}(Weak closed-loop
stability)} Let the system model be specified such that $\mathbf{\delta}_{\mathbf{u}}\equiv\mathbf{0}_{n}$.
The closed-loop system described by (\ref{eq: final-tracking-error-dynamics})
is stable in the sense that the state tracking error is bounded, if
the following three conditions are satisfied: 1) the desired error
dynamics, $\dot{\mathbf{e}}=\mathbf{h}(t,\,\mathbf{e})$, is globally
exponentially stable at the origin; 2) the function $\mathbf{h}(t,\,\mathbf{e})$
is continuously differentiable, and there exists a positive scalar
$l_{\mathbf{h}}$ such that $\left\Vert \mathbf{h}(t,\,\mathbf{e}_{1})-\mathbf{h}(t,\,\mathbf{e}_{2})\right\Vert \le l_{\mathbf{h}}\left\Vert \mathbf{e}_{1}-\mathbf{e}_{2}\right\Vert $
for any $t\ge t_{0}$ and $\mathbf{e}_{1},\,\mathbf{e}_{2}$ in the
admissible domain; 3) the disturbance, measurement noise and observer
gain satisfy the conditions in Lemma \ref{lm: state observation error bound}.
More precisely, under these conditions there exists a finite time
$T_{2}$ ($\ge T_{1}\ge t_{0}$) such that the tracking error is bounded
as{

\vspace{-10pt}
{\small{}
\begin{equation}
\left\Vert \mathbf{e}\right\Vert \le c\left\Vert \xi\right\Vert \le\frac{c\sigma_{\max}(\mathbf{P})}{\sigma_{\min}(\mathbf{Q})}(c_{\dot{\mathbf{w}}}+c_{\mathbf{v}}\left\Vert \mathbf{L}\Pi\right\Vert )(l_{\mathbf{h}}+\left\Vert [\mathbf{A}\,\,\Gamma]\right\Vert ),\label{eq: discrepancy_bound}
\end{equation}
}for all $t\ge T_{2}$ and some constant $c$, where the positive
definite matrices $\{\mathbf{P},\thinspace\thinspace\mathbf{Q}\}$
and the constants $\{c_{\dot{\mathbf{w}}},\thinspace\thinspace c_{\mathbf{v}}\}$
are defined in Lemma \ref{lm: state observation error bound}. }{\small \par}
\end{thm}
\begin{proof}
Conditions 1) and 2) imply that the tracking error dynamics described
in (\ref{eq: final-tracking-error-dynamics}) is input-to-state stable
with the input being the total design error $\xi$ (Lemma 4.6 in \citep{khalil2002nonlinear}).
To prove the theorem, it is sufficient to show that $\xi$ is bounded.
Condition 2) implies that $\left\Vert \mathbf{\mathbf{\delta}}_{\mathbf{h}}\right\Vert =\left\Vert \mathbf{h}(t,\,\mathbf{e})-\mathbf{h}(t,\,\hat{\mathbf{e}})\right\Vert \le l_{\mathbf{h}}\left\Vert \mathbf{\mathbf{\delta}}_{\mathbf{x}}\right\Vert \le l_{\mathbf{h}}\left\Vert \mathbf{\mathbf{\delta}}_{\bar{\mathbf{x}}}\right\Vert $.
With $\xi=-[\mathbf{A}\,\,\Gamma]\mathbf{\mathbf{\delta}}_{\bar{\mathbf{x}}}-\mathbf{\mathbf{\delta}}_{\mathbf{h}}$
and condition 3), it follows from Lemma \ref{lm: state observation error bound}
that there exists a finite time $T_{1}$ ($\ge t_{0}$) such that
$\left\Vert \xi\right\Vert \le\frac{\sigma_{\max}(\mathbf{P})}{\sigma_{\min}(\mathbf{Q})}(c_{\dot{\mathbf{w}}}+c_{\mathbf{v}}\left\Vert \mathbf{L}\Pi\right\Vert )(l_{\mathbf{h}}+\left\Vert [\mathbf{A}\,\,\Gamma]\right\Vert )$
for all $t\ge T_{1}$. Then, the specific bound of the tracking error
($\bar{\mathbf{e}}$) can be derived by referring to the proof of
Lemma 4.6 in \citep{khalil2002nonlinear}, which concludes that there
exists a finite time $T_{2}$ ($\ge T_{1}\ge t_{0}$) such that $\left\Vert \mathbf{e}\right\Vert \le c\left\Vert \xi\right\Vert $
for all $t\ge T_{2}$ for some constant $c$.
\end{proof}
{
The result implies that a system model is better if it enables a smaller
total design error $\xi$, which is a synergistic result of the used system model, the extended state observer,
and the controller. 
}

The assumption of $\mathbf{\delta}_{\mathbf{u}}\equiv\mathbf{0}_{n}$
imposes certain restrictions on the system model. Also the assumption
of bounded $\dot{\mathbf{w}}$ and $\mathbf{v}$ may not hold if the
two vectors depend on the control applied. For these reasons, it is
desirable to prove the stability without making these assumptions.

To that end, let us specify the reference error dynamics as: $\dot{\mathbf{e}}=\mathbf{h}(t,\,\mathbf{e}):=\mathbf{Ke}$,
where $\mathbf{K}$ is Hurwitz. Define $\bar{\mathbf{e}}=[\mathbf{\delta}_{\bar{\mathbf{x}}}^{\top}\thinspace\thinspace\mathbf{e}^{\top}]^{\top}$,
which collects the estimation error of the extended state and the
tracking error of the original state, and define $\tilde{\mathbf{B}}=\mathbf{I}_{n}-\mathbf{B}\mathbf{B}^{\dagger}$.
The extended error dynamics are then derived as follows (refer to
\ref{sec:Deriving-the-closed-loop}): 
\begin{equation}
\dot{\bar{\mathbf{e}}}=(\mathbf{H}+\Delta_{t})\bar{\mathbf{e}}+\mathbf{\delta}_{t},\label{eq: estimation and tracking error dynamics}
\end{equation}
where the matrices $\mathbf{H}$ and $\Delta_{t}$, and the vector
$\delta_{t}$ are defined as in \eqref{eq: H}-\eqref{eq: delta}.
The subscript $t$ indicates that $\Delta_{t}$ and $\delta_{t}$
may be time-varying. The closed-loop stability can then be established
without making the aforementioned assumptions. 

\begin{figure*}
\begin{align}
\mathbf{H}= & \left[\begin{array}{cc}
\tilde{\mathbf{A}} & \mathbf{O}_{(n+k)\times n}\\
-\mathbf{B}\mathbf{B}^{\dagger}[\mathbf{A}-\mathbf{K}\,\,\Gamma] & \mathbf{A}+\mathbf{B}\mathbf{B}^{\dagger}(\mathbf{K}-\mathbf{A})
\end{array}\right],\thinspace\thinspace\Delta_{t}=\left[\begin{array}{cc}
\mathbf{O}_{n\times(n+k)} & \mathbf{O}_{n}\\
\frac{\partial\mathbf{w}}{\partial\hat{\mathbf{u}}}\mathbf{B}^{\dagger}[\begin{array}{cc}
\mathbf{A}-\mathbf{K} & \Gamma\end{array}] & \frac{\partial\mathbf{w}}{\partial\hat{\mathbf{u}}}\mathbf{B}^{\dagger}(\mathbf{A}-\mathbf{K})-\frac{\partial\mathbf{w}}{\partial\mathbf{x}}\mathbf{A}\\
\mathbf{O}_{n\times(n+k)} & \mathbf{O}_{n}
\end{array}\right]\label{eq: H}\\
\delta_{t}= & \left[\begin{array}{c}
\left[\begin{array}{c}
\mathbf{0}_{n}\\
\left(\frac{\partial\mathbf{w}}{\partial\mathbf{x}}-\frac{\partial\mathbf{w}}{\partial\hat{\mathbf{u}}}\mathbf{B}^{\dagger}\right)\mathbf{A}\mathbf{x}_{r}+\frac{\partial\mathbf{w}}{\partial\hat{\mathbf{u}}}\mathbf{B}^{\dagger}\dot{\mathbf{x}}_{r}+\frac{\partial\mathbf{w}}{\partial t}+\frac{\partial\mathbf{w}}{\partial\mathbf{w}_{0}}\dot{\mathbf{w}}_{0}+\frac{\partial\mathbf{w}}{\partial\mathbf{x}}\mathbf{B}\hat{\mathbf{u}}+\left(\frac{\partial\mathbf{w}}{\partial\mathbf{x}}-\frac{\partial\mathbf{w}}{\partial\hat{\mathbf{u}}}\mathbf{B}^{\dagger}\right)\Gamma\mathbf{w}
\end{array}\right]-\mathbf{L}\Pi\mathbf{v}\\
\tilde{\mathbf{B}}(\dot{\mathbf{x}}_{r}-\mathbf{Ax}_{r})-\tilde{\mathbf{B}}\Gamma\mathbf{w}
\end{array}\right].\label{eq: delta}
\end{align}
\end{figure*}

\begin{thm}
\noindent \textup{\label{thm: Closed-loop-stability}(Closed-loop
stability)} The closed-loop system described by (\ref{eq: estimation and tracking error dynamics})
is stable in the sense that the estimation and tracking errors are
bounded if the following conditions are satisfied:\\
 1) the model mismatch functions $\mathbf{w}$, $\mathbf{v}$, and
the partial derivatives of $\mathbf{w}$ satisfy the following inequalities,

\noindent 
\begin{gather*}
\left\Vert \mathbf{w}(t,\,\mathbf{x},\,\mathbf{u},\,\mathbf{w}_{0})\right\Vert \le l_{\mathbf{w}}+l_{\mathbf{w}}^{\mathbf{x}}\left\Vert \mathbf{x}\right\Vert +l_{\mathbf{w}}^{\mathbf{u}}\left\Vert \mathbf{u}\right\Vert +l_{\mathbf{w}}^{\mathbf{w}0}\left\Vert \mathbf{w}_{0}\right\Vert ,\\
\left\Vert \dfrac{\partial\mathbf{w}(t,\,\mathbf{x},\,\mathbf{u},\,\mathbf{w}_{0})}{\partial t}\right\Vert \le l_{\partial\mathbf{w}},\thinspace\thinspace\left\Vert \dfrac{\partial\mathbf{w}(t,\,\mathbf{x},\,\mathbf{u},\,\mathbf{w}_{0})}{\partial\mathbf{x}}\right\Vert \le l_{\partial\mathbf{w}}^{\mathbf{x}},\\
\left\Vert \frac{\partial\mathbf{w}(t,\,\mathbf{x},\,\mathbf{u},\,\mathbf{w}_{0})}{\partial\mathbf{u}}\right\Vert \le l_{\partial\mathbf{w}}^{\mathbf{u}},\thinspace\thinspace\left\Vert \dfrac{\partial\mathbf{w}(t,\,\mathbf{x},\,\mathbf{u},\,\mathbf{w}_{0})}{\partial\mathbf{w}_{0}}\right\Vert \le l_{\partial\mathbf{w}}^{\mathbf{w}_{0}},\\
\left\Vert \mathbf{v}(t,\,\mathbf{x},\,\mathbf{u},\,\mathbf{v}_{0})\right\Vert \le l_{\mathbf{v}}+l_{\mathbf{v}}^{\mathbf{x}}\left\Vert \mathbf{x}\right\Vert +l_{\mathbf{v}}^{\mathbf{u}}\left\Vert \mathbf{u}\right\Vert +l_{\mathbf{v}}^{\mathbf{v}0}\left\Vert \mathbf{v}_{0}\right\Vert 
\end{gather*}

\noindent for all $t\ge t_{0}$, where $l_{\mathbf{w}}$, $l_{\mathbf{w}}^{\mathbf{x}/\mathbf{u}/\mathbf{w}_{0}}$,
$l_{\partial\mathbf{w}}$, $l_{\partial\mathbf{w}}^{\mathbf{x}/\mathbf{u}/\mathbf{w}_{0}}$,
$l_{\mathbf{v}}$ and $l_{\mathbf{v}}^{\mathbf{x}/\mathbf{u}/\mathbf{v}_{0}}$
are non-negative constants;\\
 2) the reference state and its derivative are bounded for all $t\ge t_{0}$,
i.e., $\left\Vert \mathbf{x}_{r}\right\Vert \le c_{\mathbf{x}r}$
and $\left\Vert \dot{\mathbf{x}}_{r}\right\Vert \le c_{\dot{\mathbf{x}}r}$
for certain positive constants $c_{\mathbf{x}r}$ and $c_{\dot{\mathbf{x}}r}$;
\\
 3) the original disturbance and its derivative are bounded for all
$t\ge t_{0}$, i.e., $\left\Vert \mathbf{w}_{0}\right\Vert \le c_{\mathbf{w}0}$
and $\left\Vert \dot{\mathbf{w}}_{0}\right\Vert \le c_{\dot{\mathbf{w}}0}$
for some non-negative constants $c_{\mathbf{w}0}$ and $c_{\dot{\mathbf{w}}0}$,
and the original noise is bounded for all $t\ge t_{0}$, i.e., $\left\Vert \mathbf{v}_{0}\right\Vert \le c_{\mathbf{v}0}$
for a certain non-negative constant $c_{\mathbf{v}0}$, and so is
the control action, that is, $\left\Vert \hat{\mathbf{u}}\right\Vert \le c_{\mathbf{u}}$
for some positive constant $c_{\mathbf{u}}$; \\
{
 4) the observer gain matrix $\mathbf{L}$ is such that $\tilde{\mathbf{A}}$
is Hurwitz, and the gain matrix $\mathbf{K}$ of the desired tracking
error dynamics is such that $\mathbf{A}+\mathbf{B}\mathbf{B}^{\dagger}(\mathbf{K}-\mathbf{A})$
is Hurwitz, and meanwhile the following condition is satisfied
\begin{equation}
2\mathbf{M}-\Delta{}_{t}^{\top}\mathbf{N}-\mathbf{N}\Delta_{t}-2\beta_{1}\sigma_{\max}(\mathbf{N})\mathbf{I}_{2n+k}\succ\mathbf{O}_{2n+k},\label{eq: key condition for thm 2}
\end{equation}
for some $\mathbf{N}\succ\mathbf{O}_{2n+k}$ which is the solution
of the Lyapunov equation $\mathbf{H}^{\top}\mathbf{N}+\mathbf{N}\mathbf{H}=-2\mathbf{M}$
with a matrix $\mathbf{M}\succ\mathbf{O}_{2n+k}$. Here, $\beta_{1}:=l_{\mathbf{v}}^{\mathbf{x}}\left\Vert \mathbf{L}\Pi\right\Vert +l_{\mathbf{w}}^{\mathbf{x}}(l_{\partial\mathbf{w}}^{\mathbf{x}}\left\Vert \Gamma\right\Vert +l_{\partial\mathbf{w}}^{\mathbf{u}}\left\Vert \mathbf{B}^{\dagger}\Gamma\right\Vert +\left\Vert \tilde{\mathbf{B}}\Gamma\right\Vert )$.
}%
\end{thm}
\begin{proof}
Consider the closed-loop system $\dot{\bar{\mathbf{e}}}=\mathbf{H}\bar{\mathbf{e}}+\Delta_{t}\bar{\mathbf{e}}+\delta_{t}$.
Since $\tilde{\mathbf{A}}$ and $\mathbf{A}+\mathbf{B}\mathbf{B}^{\dagger}(\mathbf{K}-\mathbf{A})$
are both Hurwitz, it follows that $\mathbf{H}$ is also Hurwitz by
definition. Given $\mathbf{M}\succ\mathbf{O}_{2n+k}$, introduce the
Lyapunov function $V(\bar{\mathbf{e}})=\bar{\mathbf{e}}^{\top}\mathbf{N}\bar{\mathbf{e}}$,
where the positive definite matrix $\mathbf{N}$ is a unique solution
to the Lyapunov equation $\mathbf{H}^{\top}\mathbf{N}+\mathbf{N}\mathbf{H}=-2\mathbf{M}$.
Then the derivative of $V(t,\,\mathbf{e})$ is deduced as in \eqref{eq: proof}
in the Appendix. By condition 4), $\dot{V}<0$ if $\left\Vert \bar{\mathbf{e}}\right\Vert $
is large enough (the worst-case threshold value of which can be computed
from the last inequality of \eqref{eq: proof}). This implies that
the error $\bar{\mathbf{e}}$ will be bounded, and hence completes
the proof.
\end{proof}
\textcolor{black}{Condition 1) of Theorem \ref{thm: Closed-loop-stability}
requires the model mismatches not to change too fast, so that it is
possible to suppress these mismatches by compensation and control.
Condition 2) on the bounded reference state and its derivative can
be satisfied via an appropriate design of the reference system. Condition
3) is not restrictive since unbounded control or disturbance is not
allowed in normal operations. Condition 4) is thus a key requirement
to enable the closed-loop stability, which roughly means that } 
\begin{itemize}
\item \textcolor{black}{on one hand, the observer and the reference tracking
error (as governed by $\mathbf{H}$) should have fast convergent dynamics
such that $\frac{\sigma_{\min}(\mathbf{M})}{\sigma_{\max}(\mathbf{N})}$
is large (note that $\mathbf{N}=2\int_{t_{0}}^{\infty}e^{\mathbf{H}^{\top}t}\mathbf{M}e^{\mathbf{H}t}dt$
is an explicit solution to the Lyapunov equation \citep{Chen1999}),
and } 
\item \textcolor{black}{on the other hand, the model mismatches (as encoded
by $\mathbf{w}$ and $\mathbf{v}$) should have slow dynamics such
that the factor $\beta_{1}$, i.e, }\textcolor{black}{\small{}$\begin{array}{c}
l_{\mathbf{v}}^{\mathbf{x}}||\mathbf{L}\Pi||+l_{\mathbf{w}}^{\mathbf{x}}(l_{\partial\mathbf{w}}^{\mathbf{x}}||\Gamma||+l_{\partial\mathbf{w}}^{\mathbf{u}}||\mathbf{B}^{\dagger}\Gamma||+||\tilde{\mathbf{B}}\Gamma||)\end{array}$}\textcolor{black}{, is small. } 
\end{itemize}
Since a gain matrix $\mathbf{L}$ enlarging \textcolor{black}{$\frac{\sigma_{\min}(\mathbf{M})}{\sigma_{\max}(\mathbf{N})}$}
may inflate $\left\Vert \mathbf{L}\Pi\right\Vert $ at the same time,
the above conditions indicate an intrinsic trade-off in the observer
design.
 
{
We remark that the reference dynamics can only be approximately achieved in the control design due to the various errors involved, and that the stability of the closed-loop system does not lie in the exact matching with the reference specifications.
}%
We also remark that the derived stability conditions rely on the assumption
of using a first-order ESO. The conditions may be relaxed if a higher-order
ESO which treats the higher-order derivative(s) of $\mathbf{w}$ as
additional state(s) is employed to estimate the disturbance \citep{johnson1975observers,Miklosovic2006,Madonski2013}.
Moreover, to deal with measurement noise, the measurements can be filtered
before use if the cost incurred is mild compared to the benefit.

\section{Numerical examples \label{sec:Numerical-example}}

Firstly, a first-order system with uncertainty is used to illustrate
the proposed control design and validate the stability conditions
presented in Theorem \ref{thm: Closed-loop-stability}. Then, the
proposed approach is applied to design controllers for an inverted
pendulum, and the control performances subjected to different levels of model mismatches are
examined via simulations.

\subsection{Control of a first-order uncertain system} \label{subsec: example1}

Consider the following first-order system model in integral canonical
form:
\begin{equation}
\dot{x}=2x+3u+w,\thinspace\thinspace y=x+v,\thinspace\thinspace t\ge0,\label{eq: example 1}
\end{equation}
where $x$ and $u$ are the state and control, respectively, and $w$
and $v$ are the unknown parts of the model. Suppose that $w=0.2x+0.3u+w_{0}$,
where $w_{0}=0.1\sin t$. Thus we have, $l_{w}=l_{\partial w}=0$,
$l_{w}^{x}=l_{\partial w}^{x}=0.2$, $l_{w}^{u}=l_{\partial w}^{u}=0.3$,
$l_{w}^{w0}=l_{\partial w}^{w0}=1$, $c_{w0}=c_{\dot{w}0}=0.1$. And
suppose that $v=v_{0}$, where $v_{0}$ is a zero-mean Gaussian noise
with a variance equal to $0.01$ and the value of $v_{0}$ is truncated
to the range of {[}-0.1, 0.1{]}. Then we have, $l_{v}=l_{v}^{x}=l_{v}^{u}=0$,
$l_{v}^{v0}=1$ and $c_{v0}=0.1$. The goal is to design the control
$u$, satisfying $|u|\le c_{u}:=5$, based on the measurement $y$
such that the state $x$ is stabilized at the point of $1$.

Specify the reference state dynamics as: $\dot{x}_{r}=-kx_{r}+ku_{r}$,
where $u_{r}$ is a unit step signal and $k$ is a given positive
number. And specify the reference state tracking error dynamics as:
$\dot{e}_{r}=-ke_{r}$. Then, the desired control can be estimated
as
\begin{equation}
\hat{u}=[ku_{r}-(k+2)\hat{x}-\hat{w}]/3.\label{eq: control for example 1}
\end{equation}
Define $\hat{\bar{\mathbf{x}}}=[\hat{x},\thinspace\thinspace\hat{w}]^{\top}$.
The two estimates are obtained from the observer: $\dot{\hat{\bar{\mathbf{x}}}}=\left[\begin{array}{cc}
2 & 1\\
0 & 0
\end{array}\right]\hat{\bar{\mathbf{x}}}+\left[\begin{array}{c}
3\\
0
\end{array}\right]u+\bar{\mathbf{l}}(y-\hat{y})$, $\hat{y}=[1,\thinspace\thinspace0]\hat{\bar{\mathbf{x}}}$, where
$\bar{\mathbf{l}}$ is the $2\times1$ observer gain vector. The
matrices used in the stability analysis are

\begin{align*}
\tilde{\mathbf{A}} & =\left[\begin{array}{cc}
2 & 1\\
0 & 0
\end{array}\right]-\bar{\mathbf{l}}[1,\thinspace\thinspace0],\thinspace\thinspace\mathbf{H}=\left[\begin{array}{cc}
\tilde{\mathbf{A}} & \mathbf{0}_{2}\\
-[k+2\,\,1] & -k
\end{array}\right],\\
\Delta & =0.1\cdot\left[\begin{array}{ccc}
0 & 0 & 0\\
k+2 & 1 & k-2\\
0 & 0 & 0
\end{array}\right],\thinspace\thinspace\Gamma=\Pi=1.
\end{align*}
By Theorem \ref{thm: Closed-loop-stability}, it suffices to design
the gain vector $\bar{\mathbf{l}}$ such that $\tilde{\mathbf{A}}$
is Hurwitz and $2\mathbf{M}-\Delta{}^{\top}\mathbf{N}-\mathbf{N}\Delta-0.12\sigma_{\max}(\mathbf{N})\mathbf{I}_{3}\succ\mathbf{O}_{3}$,
where $\mathbf{N}$ is the solution to the Lyapunov equation $\mathbf{H}^{\top}\mathbf{N}+\mathbf{N}\mathbf{H}=-2\mathbf{M}$
for certain $\mathbf{M}\succ\mathbf{O}_{3}$.

Set $\bar{\mathbf{l}}=[6k+2,\thinspace\thinspace9k^{2}]^{\top}$.
Then the two poles of the observer are both equal to $-3k$. With
$\mathbf{M}=\mathbf{I}_{3}$, numerical computation shows that the
aforementioned stability condition is satisfied for all $0.8\le k\le4.1$.
For instance, with $k=1.5$, the minimum eigenvalue of $2\mathbf{M}-\Delta{}^{\top}\mathbf{N}-\mathbf{N}\Delta-0.12\sigma_{\max}(\mathbf{N})\mathbf{I}_{3}$
is obtained as 0.86, which implies the positive definiteness of the
matrix. Consequently, the closed-loop system is stable by Theorem
\ref{thm: Closed-loop-stability}. This is verified by the simulation
results shown in Fig. \ref{fig: results of example 1}. Simulations also showed that a larger $k$ will lead to a faster convergence to the reference state
but at the cost of a more oscillating control to counteract the more
serious effect of measurement noise, and that
the obtained range of $k$ is sufficient but not necessary for the
closed-loop stability. These results are not shown due to page limit. 
\noindent \begin{center}
\begin{figure}
\noindent \begin{centering}
\includegraphics[scale=0.48]{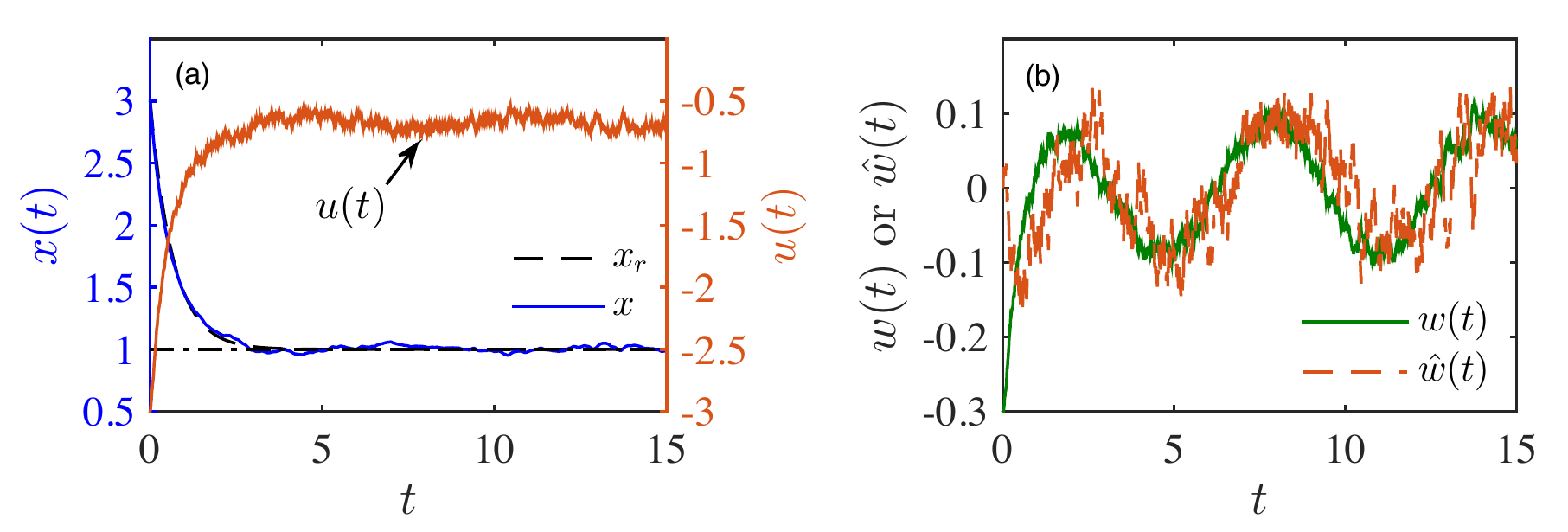} 
\par\end{centering}
\noindent \centering{}\caption{\label{fig: results of example 1}Trajectories of the state, control
and total disturbance.}
\end{figure}
\par\end{center}

\subsection{Control of an inverted pendulum}

Consider a normalized model of the pendulum when the control input
is the acceleration of the pivot \citep{Astrom2008}:
\begin{equation}
\dot{x}_{1}=x_{2},\,\,\dot{x}_{2}=\sin x_{1}-u\cos x_{1}+w_{0},\label{eq: pendulum_system}
\end{equation}
where $x_{1}$ is the angular position of the pendulum with the origin
at the upright position, $x_{2}$ is the angular velocity of the pendulum, and $w_{0}$ is an unknown disturbance which is equal to 0 if $t\le10$
and to $\sin t$ if $t>10$. The goal is to design a controller based
on the measurable states $x_{1}$ and $x_{2}$ which stabilizes the
pendulum at the upright position.

Specify the reference closed-loop model as: $\dot{x}_{r,\,1}=x_{r,\,2}$
and $\dot{x}_{r,\,2}=-k_{1}x_{r,\,1}-k_{2}x_{r,\,2}$, where $k_{1}$
and $k_{2}$ are positive scalars, and the desired tracking error
dynamics as: $\dot{e}_{1}=e_{2}$ and $\dot{e}_{2}=-k_{1}e_{1}-k_{2}e_{2}$,
where $e_{1}:=x_{r,\,1}-x_{1}$ and $e_{2}:=x_{r,\,2}-x_{2}$. Depending
on the system model used, different controllers can be designed by
the proposed approach. 

\textit{Case A:} The design bases on the ideal model \eqref{eq: pendulum_system}
and ignores $w_{0}$, leading to the following control
\begin{equation}
u=(k_{1}x_{1}+k_{2}x_{2}+\sin x_{1}) / \cos x_{1}.\label{eq: controller A}
\end{equation}
It equals the control obtained by an input-output
linearization technique \citep{Srinivasan2009}. The singularity of
the control at $x_{1}=\frac{2k+1}{2}\pi$ for $k=0,\,1,\,2,\,...$,
can be resolved by bounding the input and meanwhile switching the
reference value of $x_{1}$ properly \citep{Srinivasan2009}. This
controller is treated as a reference controller without explicit compensation
for the unknown disturbance.

\textit{Case B:} The dynamic model of $x_{2}$ is replaced by a fictitious
model: $\dot{x}_{2}=-\alpha u+w$, where $w$ lumps any model mismatch (i.e., $w=\sin x_{1}+u(\alpha-\cos x_{1})+w_{0}$) and $\alpha$ is a given scalar which controls the mismatch.
If a Type-I estimator is applied, it leads to an estimated control:
\begin{equation}
\hat{u}=\mathcal{L}^{-1}\left(\dfrac{k_{1}X_{1}+(k_{2}+sF_{x})X_{2}}{\alpha(1-F_{u})}\right),\label{eq: controller B-I}
\end{equation}
where $F_{x}$ and $F_{u}$ are filters for estimating $\dot{x}_{2}$
and $u$, respectively. Note that the disturbance $w$ has been estimated
and compensated in an implicit manner. If a Type-II estimator is used
instead, then the estimate $\hat{w}$ is obtained from an ESO defined
in \eqref{eq: ESO} and the observer gain matrix $\bar{\mathbf{L}}\in\mathbb{R}^{3\times2}$
is designed such that $\mathbf{\bar{A}-\bar{L}\bar{C}}$ is Hurwitz.
The control then takes the form of
\begin{equation}
\hat{u}=(k_{1}x_{1}+k_{2}x_{2}+\hat{w})/\alpha.\label{eq: controller B-II}
\end{equation}

The three controllers in \eqref{eq: controller A}, \eqref{eq: controller B-I}
and \eqref{eq: controller B-II} are named as controllers A, B.I and
B.II, respectively. The design parameters of the controllers are specified
as: $k_{1}=k_{2}=2$, $\left|u_{\text{max}}\right|=5$ (bounded control),
and $F_{x}=F_{u}=1/(0.05s+1)$. The filters $F_{x}$ and $F_{u}$
are such that the derivative of state $x_{2}$ can well be estimated,
and the observer gain matrix $\bar{\mathbf{L}}$ is a function of
$\alpha$ such that the three poles of the ESO are placed at -20,
-20 and -40, which are ten or more times faster than the actual state
dynamics.

\begin{figure}
\noindent \begin{centering}
\includegraphics[scale=0.54]{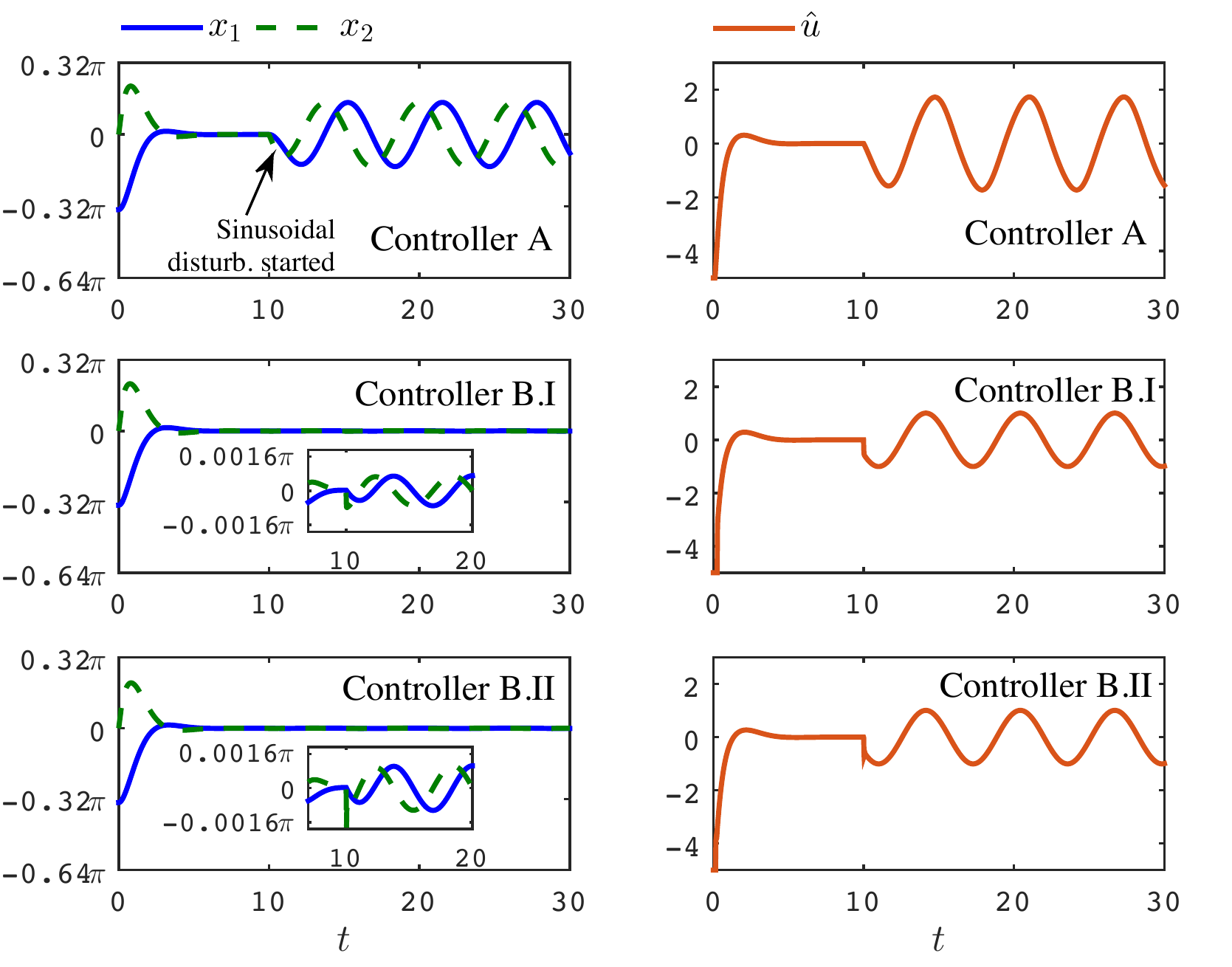}
\par\end{centering}
\noindent \centering{}\caption{\label{fig: states_control_pi_3}Performances of the three controllers.}\vspace{-6pt}
\end{figure}

With $(x_{1}(0),\,x_{2}(0))=(-\pi/3,\,0)$ and noise-free
measurements, the simulation results for $\alpha=0.1$ are shown in
Fig. \ref{fig: states_control_pi_3}. In the absence of the sinusoid disturbance, all three controllers were able to stabilize the pendulum
with comparable performances. When the sinusoid disturbance appeared,
however, controller A led to large tracking errors, in sharp contrast to controllers B.I and B.II.
Further simulations show that the small tracking errors of controllers B.I and B.II were maintained if the measurements were corrupted with additive zero-mean
Gaussian noises, while
the controls were experiencing frequent variations. The results are not shown for brevity. 

By varying the parameter $\alpha$, we simulated the control system
with different degrees of model mismatches. The performances of controllers
B.I and B.II are shown in Fig. \ref{fig: IAE - alpha}(a)-(b). The
integral absolute error (IAE) of state $x_{2}$ with respect to the
origin and the integral variation (IV) of control $\hat{u}$ are used as
the performance indices. Smaller IAE and IV indicate a better control
performance. As observed, the IAE of $x_{2}$ increases as $\alpha$
is enlarged from 0.1 to 1.0, and so does the IV of $\hat{u}$ in most of
the range, both of which indicate a degrading performance. This is
somehow counterintuitive since a better performance were expected
when the model tends to be more accurate (note that, $\alpha \hat{u}\rightarrow \hat{u}\cos x_{1}$ as $x_{1}\rightarrow0$ and $\alpha\rightarrow1$).
The underlying fact is that the affine control component, $\hat{u}(\alpha-\cos x_{1})$,
of the model mismatch $w$ acts to counteract the remaining mismatch,
$\sin x_{1}+w_{0}$. As $\alpha$ approaches 1, the factor ($\alpha-\cos x_{1}$)
approaches zero and hence the affine control component tends to be
nullified. Consequently, this weakens the counteraction action and
results in a larger total disturbance and a notable estimation error. This is reflected in the results shown in Fig. \ref{fig: IAE - alpha}(c)-(d), in which the disturbance component $(\sin x_{1}+w_{0})$ is almost unchanged as $\alpha$ is changed from 0.1 to 1.

\begin{figure}
\noindent \begin{centering}
\includegraphics[scale=0.54]{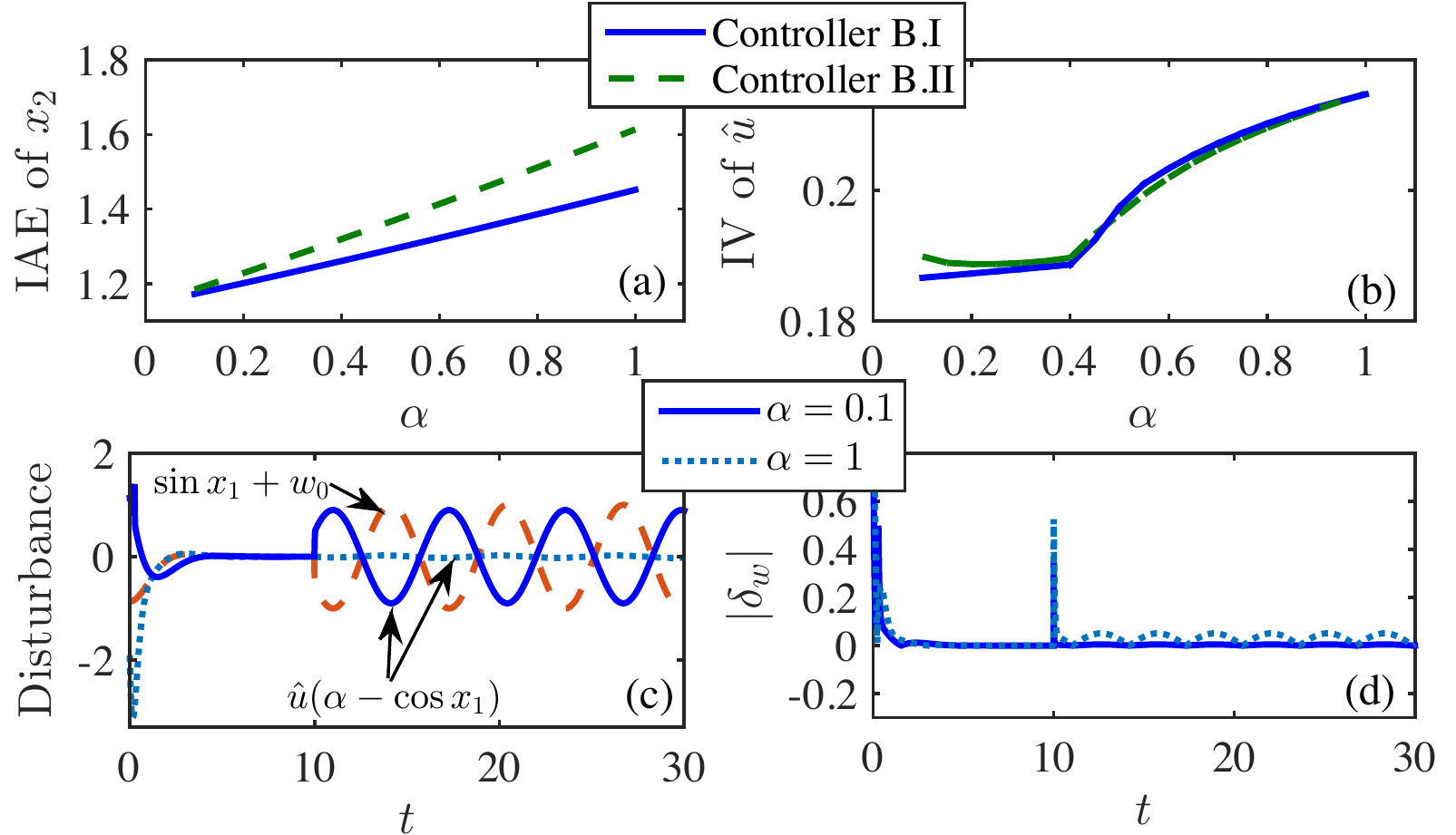}
\par\end{centering}
\noindent \centering{}\caption{\label{fig: IAE - alpha} Performances of controllers B.I and B.II
with different values of $\alpha$. (a)-(b): the performances of state
tracking and control input vs. $\alpha$; (c)-(d): the evolutions
of disturbance components and absolute disturbance estimation errors when
controller B.II was implemented. Meanwhile, the IAE of $x_{2}$ and
the IV of $u$ under controller A were obtained as 6.71 and 0.28,
respectively.}\vspace{-6pt}
\end{figure}

\section{Conclusions \label{sec:Conclusions}}

This work presented an output feedback approach for controlling a
general MIMO system to track a reference state trajectory. The approach uses a composite
observer to estimate the system states and disturbances simultaneously,
and then uses the estimates to derive the control. As the disturbances
lump all unmodeled dynamics and are compensated online, the design
admits a crude system model to be used for the control purpose. The
closed-loop stability was established under some standard conditions. 

It is desirable to analyze the closed-loop performance and extend
the design to broader classes of systems such as those with time delays.
It is also of interest to thoroughly investigate the interplay between
state control and disturbance compensation, and then optimize the
joint design. Future research may also consider embedding explicit
disturbance compensation into existing control methods, such as robust
control, optimal control, predictive control, etc., for obtaining
flexible and enhanced controls.

\section*{Acknowledgment}

This work was supported in part by the project DPI2016-78338-R as funded the Ministerio de Econom\'{i}a y Competitividad, Spain, and by the Projects of Major International (Regional) Joint Research Program (Grant no. 61720106011) as funded by NSFC, Singapore.

\appendix
\begin{figure*}
\section{Derivation of the closed-loop dynamics \label{sec:Deriving-the-closed-loop}}

Define $\tilde{\mathbf{B}}=\mathbf{I}_{n}-\mathbf{B}\mathbf{B}^{\dagger}$.
Given $\mathbf{h}(t,\,\mathbf{e})=\mathbf{Ke}$, by \eqref{eq: fictitious-LTI-model-with-extended-state}-\eqref{eq: final-tracking-error-dynamics}
the estimation error ($\mathbf{\delta}_{\bar{\mathbf{x}}}$) and the
tracking error ($\mathbf{e}$) evolve as follows:

\noindent 
\begin{align*}
\dot{\mathbf{\mathbf{\delta}}}_{\bar{\mathbf{x}}} & =\tilde{\mathbf{A}}\mathbf{\mathbf{\delta}}_{\bar{\mathbf{x}}}-\mathbf{L}\Pi\mathbf{v}+\left[\begin{array}{c}
\mathbf{0}_{n}\\
\dot{\mathbf{w}}
\end{array}\right]=\tilde{\mathbf{A}}\mathbf{\mathbf{\delta}}_{\bar{\mathbf{x}}}-\mathbf{L}\Pi\mathbf{v}+\left[\begin{array}{c}
\mathbf{0}_{n}\\
\frac{\partial\mathbf{w}}{\partial t}+\frac{\partial\mathbf{w}}{\partial\mathbf{x}}\dot{\mathbf{x}}+\frac{\partial\mathbf{w}}{\partial\hat{\mathbf{u}}}\dot{\hat{\mathbf{u}}}+\frac{\partial\mathbf{w}}{\partial\mathbf{w}_{0}}\dot{\mathbf{w}}_{0}
\end{array}\right]\\
 & =\tilde{\mathbf{A}}\mathbf{\mathbf{\delta}}_{\bar{\mathbf{x}}}-\mathbf{L}\Pi\mathbf{v}+\left[\begin{array}{c}
\mathbf{0}_{n}\\
\left(\begin{array}{c}
\frac{\partial\mathbf{w}}{\partial t}+\frac{\partial\mathbf{w}}{\partial\mathbf{w}_{0}}\dot{\mathbf{w}}_{0}+\frac{\partial\mathbf{w}}{\partial\mathbf{x}}(\mathbf{A}\mathbf{x}+\mathbf{B}\hat{\mathbf{u}}+\Gamma\mathbf{w})+\frac{\partial\mathbf{w}}{\partial\hat{\mathbf{u}}}\mathbf{B}^{\dagger}\left(\dot{\mathbf{x}}_{r}-\Gamma\hat{\mathbf{w}}-\mathbf{A}\hat{\mathbf{x}}-\mathbf{K}\hat{\mathbf{e}}\right)\end{array}\right)
\end{array}\right]\\
 & =\tilde{\mathbf{A}}\mathbf{\mathbf{\delta}}_{\bar{\mathbf{x}}}-\mathbf{L}\Pi\mathbf{v}+\left[\begin{array}{c}
\mathbf{0}_{n}\\
\left(\begin{array}{c}
\frac{\partial\mathbf{w}}{\partial t}+\frac{\partial\mathbf{w}}{\partial\mathbf{w}_{0}}\dot{\mathbf{w}}_{0}+\frac{\partial\mathbf{w}}{\partial\mathbf{x}}(\mathbf{A}\mathbf{x}+\mathbf{B}\hat{\mathbf{u}}+\Gamma\mathbf{w})\\
+\frac{\partial\mathbf{w}}{\partial\hat{\mathbf{u}}}\mathbf{B}^{\dagger}[\begin{array}{cc}
\mathbf{A}-\mathbf{K} & \Gamma\end{array}]\mathbf{\mathbf{\delta}}_{\bar{\mathbf{x}}}+\frac{\partial\mathbf{w}}{\partial\hat{\mathbf{u}}}\mathbf{B}^{\dagger}\left(\dot{\mathbf{x}}_{r}-\Gamma\mathbf{w}-\mathbf{A}\mathbf{x}-\mathbf{Ke}\right)
\end{array}\right)
\end{array}\right]\\
 & =\tilde{\mathbf{A}}\mathbf{\mathbf{\delta}}_{\bar{\mathbf{x}}}-\mathbf{L}\Pi\mathbf{v}+\left[\begin{array}{c}
\mathbf{0}_{n}\\
\left(\begin{array}{c}
\frac{\partial\mathbf{w}}{\partial\hat{\mathbf{u}}}\mathbf{B}^{\dagger}[\begin{array}{cc}
\mathbf{A}-\mathbf{K} & \Gamma\end{array}]\mathbf{\mathbf{\delta}}_{\bar{\mathbf{x}}}+\left(\frac{\partial\mathbf{w}}{\partial\hat{\mathbf{u}}}\mathbf{B}^{\dagger}(\mathbf{A}-\mathbf{K})-\frac{\partial\mathbf{w}}{\partial\mathbf{x}}\mathbf{A}\right)\mathbf{e}\\
+\frac{\partial\mathbf{w}}{\partial\hat{\mathbf{u}}}\mathbf{B}^{\dagger}\dot{\mathbf{x}}_{r}
+\left(\frac{\partial\mathbf{w}}{\partial\mathbf{x}}-\frac{\partial\mathbf{w}}{\partial\hat{\mathbf{u}}}\mathbf{B}^{\dagger}\right)\mathbf{A}\mathbf{x}_{r}+\frac{\partial\mathbf{w}}{\partial t}+\frac{\partial\mathbf{w}}{\partial\mathbf{w}_{0}}\dot{\mathbf{w}}_{0}+\frac{\partial\mathbf{w}}{\partial\mathbf{x}}\mathbf{B}\hat{\mathbf{u}}+\left(\frac{\partial\mathbf{w}}{\partial\mathbf{x}}-\frac{\partial\mathbf{w}}{\partial\hat{\mathbf{u}}}\mathbf{B}^{\dagger}\right)\Gamma\mathbf{w}
\end{array}\right)
\end{array}\right],\\
\dot{\mathbf{e}} & =\mathbf{K}\left(\mathbf{e}+[\mathbf{I}_{n}\,\,\mathbf{O}_{k}]\mathbf{\mathbf{\delta}}_{\bar{\mathbf{x}}}\right)-[\mathbf{A}\,\,\mathbf{\Gamma}]\mathbf{\mathbf{\delta}}_{\bar{\mathbf{x}}}+\tilde{\mathbf{B}}\left(\mathbf{f}_{r}-\Gamma\hat{\mathbf{w}}-\mathbf{A}\hat{\mathbf{x}}-\mathbf{K}\left(\mathbf{e}+[\mathbf{I}_{n}\,\,\mathbf{O}_{k}]\mathbf{\mathbf{\delta}}_{\bar{\mathbf{x}}}\right)\right)\\
 & =-[\mathbf{A}-\mathbf{B}\mathbf{B}^{\dagger}\mathbf{K}\,\,\Gamma]\mathbf{\mathbf{\delta}}_{\bar{\mathbf{x}}}+\mathbf{B}\mathbf{B}^{\dagger}\mathbf{K}\mathbf{e}+\tilde{\mathbf{B}}\left(\mathbf{f}_{r}-\Gamma\hat{\mathbf{w}}-\mathbf{A}(\mathbf{x}_{r}-\mathbf{e}-[\mathbf{I}_{n}\,\,\mathbf{O}_{k}]\mathbf{\mathbf{\delta}}_{\bar{\mathbf{x}}})\right)\\
 & =-[\mathbf{B}\mathbf{B}^{\dagger}(\mathbf{A}-\mathbf{K})\,\,\Gamma]\mathbf{\mathbf{\delta}}_{\bar{\mathbf{x}}}+\left(\mathbf{A}+\mathbf{B}\mathbf{B}^{\dagger}(\mathbf{K}-\mathbf{A})\right)\mathbf{e}+\tilde{\mathbf{B}}\left(\mathbf{f}_{r}-\mathbf{A}\mathbf{x}_{r}\right)-\tilde{\mathbf{B}}\left(\Gamma\mathbf{w}-[\mathbf{O}_{n}\,\,\Gamma]\mathbf{\delta}_{\bar{\mathbf{x}}}\right)\\
 & =-\mathbf{B}\mathbf{B}^{\dagger}[\mathbf{A}-\mathbf{K}\,\,\Gamma]\mathbf{\mathbf{\delta}}_{\bar{\mathbf{x}}}+\left(\mathbf{A}+\mathbf{B}\mathbf{B}^{\dagger}(\mathbf{K}-\mathbf{A})\right)\mathbf{e}+\tilde{\mathbf{B}}\left(\dot{\mathbf{x}}_{r}-\mathbf{Ax}_{r}-\Gamma\mathbf{w}\right).
\end{align*}

\noindent The two equations can be written concisely as (\ref{eq: estimation and tracking error dynamics}).

\section{Deduction of the $\dot{\mathbf{V}}(t)$}

\begin{align}
\dot{V}= & \bar{\mathbf{e}}^{\top}\left((\mathbf{H}+\Delta_{t})^{\top}\mathbf{N}+\mathbf{N}(\mathbf{H}+\Delta_{t})\right)\bar{\mathbf{e}}+2\bar{\mathbf{e}}^{\top}\mathbf{N}\delta_{t}=\bar{\mathbf{e}}^{\top}\left(-2\mathbf{M}+\Delta{}_{t}^{\top}\mathbf{N}+\mathbf{N}\Delta_{t}\right)\bar{\mathbf{e}}+2\bar{\mathbf{e}}^{\top}\mathbf{N}\delta_{t}\nonumber \\
\le & -\bar{\mathbf{e}}^{\top}(2\mathbf{M}-\Delta{}_{t}^{\top}\mathbf{N}-\mathbf{N}\Delta_{t})\bar{\mathbf{e}}+2\sigma_{\max}(\mathbf{N})\left\Vert \bar{\mathbf{e}}\right\Vert \left\Vert \delta_{t}\right\Vert \nonumber \\
\le & -\bar{\mathbf{e}}^{\top}(2\mathbf{M}-\Delta{}_{t}^{\top}\mathbf{N}-\mathbf{N}\Delta_{t})\bar{\mathbf{e}}\nonumber \\
 & +2\sigma_{\max}(\mathbf{N})\left\Vert \bar{\mathbf{e}}\right\Vert \cdot\left(\begin{array}{c}
\left\Vert \mathbf{L}\Pi\right\Vert \left\Vert \mathbf{v}\right\Vert +\left\Vert \frac{\partial\mathbf{w}}{\partial\hat{\mathbf{u}}}\right\Vert \left\Vert \mathbf{B}^{\dagger}\right\Vert \left\Vert \dot{\mathbf{x}}_{r}\right\Vert +\left(\left\Vert \frac{\partial\mathbf{w}}{\partial\mathbf{x}}\right\Vert \left\Vert \mathbf{A}\right\Vert +\left\Vert \frac{\partial\mathbf{w}}{\partial\hat{\mathbf{u}}}\right\Vert \left\Vert \mathbf{B}^{\dagger}\mathbf{A}\right\Vert \right)\cdot\left\Vert \mathbf{x}_{r}\right\Vert \\
 +\left\Vert \frac{\partial\mathbf{w}}{\partial t}\right\Vert
+\left\Vert \frac{\partial\mathbf{w}}{\partial\mathbf{w}_{0}}\right\Vert \left\Vert \dot{\mathbf{w}}_{0}\right\Vert +\left\Vert \frac{\partial\mathbf{w}}{\partial\mathbf{x}}\right\Vert \left\Vert \mathbf{B}\right\Vert \left\Vert \hat{\mathbf{u}}\right\Vert +\left(\left\Vert \frac{\partial\mathbf{w}}{\partial\mathbf{x}}\right\Vert \left\Vert \Gamma\right\Vert +\left\Vert \frac{\partial\mathbf{w}}{\partial\hat{\mathbf{u}}}\right\Vert \left\Vert \mathbf{B}^{\dagger}\Gamma\right\Vert \right)\cdot\left\Vert \mathbf{w}\right\Vert \\
+\left\Vert \tilde{\mathbf{B}}\right\Vert \left\Vert \dot{\mathbf{x}}_{r}\right\Vert +\left\Vert \tilde{\mathbf{B}}\mathbf{A}\right\Vert \left\Vert \mathbf{x}_{r}\right\Vert +\left\Vert \tilde{\mathbf{B}}\Gamma\right\Vert \left\Vert \mathbf{w}\right\Vert 
\end{array}\right)\nonumber \\
\le & -\bar{\mathbf{e}}^{\top}(2\mathbf{M}-\Delta{}_{t}^{\top}\mathbf{N}-\mathbf{N}\Delta_{t})\bar{\mathbf{e}}\nonumber \\
 & +2\sigma_{\max}(\mathbf{N})\left\Vert \bar{\mathbf{e}}\right\Vert \cdot\left(\begin{array}{c}
\left\Vert \mathbf{L}\Pi\right\Vert \cdot\left(l_{\mathbf{v}}+l_{\mathbf{v}}^{\mathbf{x}}\left\Vert \mathbf{x}\right\Vert +l_{\mathbf{v}}^{\mathbf{u}}\left\Vert \hat{\mathbf{u}}\right\Vert +l_{\mathbf{v}}^{\mathbf{v}0}\left\Vert \mathbf{v}_{0}\right\Vert \right)+\left(l_{\partial\mathbf{w}}^{\mathbf{u}}\left\Vert \mathbf{B}^{\dagger}\right\Vert +\left\Vert \tilde{\mathbf{B}}\right\Vert \right)\cdot\left\Vert \dot{\mathbf{x}}_{r}\right\Vert \\
+\left(l_{\partial\mathbf{w}}^{\mathbf{x}}\left\Vert \mathbf{A}\right\Vert +l_{\partial\mathbf{w}}^{\mathbf{u}}\left\Vert \mathbf{B}^{\dagger}\mathbf{A}\right\Vert +\left\Vert \tilde{\mathbf{B}}\mathbf{A}\right\Vert \right)\cdot\left\Vert \mathbf{x}_{r}\right\Vert +l_{\partial\mathbf{w}}^{\mathbf{w}_{0}}\left\Vert \dot{\mathbf{w}}_{0}\right\Vert +l_{\partial\mathbf{w}}^{\mathbf{x}}\left\Vert \mathbf{B}\right\Vert \left\Vert \hat{\mathbf{u}}\right\Vert +l_{\partial\mathbf{w}}\\
+\left(l_{\partial\mathbf{w}}^{\mathbf{x}}\left\Vert \Gamma\right\Vert +l_{\partial\mathbf{w}}^{\mathbf{u}}\left\Vert \mathbf{B}^{\dagger}\Gamma\right\Vert +\left\Vert \tilde{\mathbf{B}}\Gamma\right\Vert \right)\cdot\left(l_{\mathbf{w}}+l_{\mathbf{w}}^{\mathbf{x}}\left\Vert \mathbf{x}\right\Vert +l_{\mathbf{w}}^{\mathbf{u}}\left\Vert \hat{\mathbf{u}}\right\Vert +l_{\mathbf{w}}^{\mathbf{w}0}\left\Vert \mathbf{w}_{0}\right\Vert \right)
\end{array}\right)\label{eq: proof}\\
\le & -\bar{\mathbf{e}}^{\top}(2\mathbf{M}-\Delta{}_{t}^{\top}\mathbf{N}-\mathbf{N}\Delta_{t})\bar{\mathbf{e}}\nonumber \\
 & +2\sigma_{\max}(\mathbf{N})\left\Vert \bar{\mathbf{e}}\right\Vert \cdot\left(\begin{array}{c}
\left\Vert \mathbf{L}\Pi\right\Vert \cdot\left(l_{\mathbf{v}}+l_{\mathbf{v}}^{\mathbf{x}}\left\Vert \mathbf{x}_{r}\right\Vert +l_{\mathbf{v}}^{\mathbf{x}}\left\Vert \mathbf{e}\right\Vert +l_{\mathbf{v}}^{\mathbf{u}}\left\Vert \hat{\mathbf{u}}\right\Vert +l_{\mathbf{v}}^{\mathbf{v}0}\left\Vert \mathbf{v}_{0}\right\Vert \right)+\left(l_{\partial\mathbf{w}}^{\mathbf{u}}\left\Vert \mathbf{B}^{\dagger}\right\Vert +\left\Vert \tilde{\mathbf{B}}\right\Vert \right)\cdot\left\Vert \dot{\mathbf{x}}_{r}\right\Vert \\
+\left(l_{\partial\mathbf{w}}^{\mathbf{x}}\left\Vert \mathbf{A}\right\Vert +l_{\partial\mathbf{w}}^{\mathbf{u}}\left\Vert \mathbf{B}^{\dagger}\mathbf{A}\right\Vert +\left\Vert \tilde{\mathbf{B}}\mathbf{A}\right\Vert \right)\cdot\left\Vert \mathbf{x}_{r}\right\Vert +l_{\partial\mathbf{w}}^{\mathbf{w}_{0}}\left\Vert \dot{\mathbf{w}}_{0}\right\Vert +l_{\partial\mathbf{w}}^{\mathbf{x}}\left\Vert \mathbf{B}\right\Vert \left\Vert \hat{\mathbf{u}}\right\Vert +l_{\partial\mathbf{w}}\\
+\left(l_{\partial\mathbf{w}}^{\mathbf{x}}\left\Vert \Gamma\right\Vert +l_{\partial\mathbf{w}}^{\mathbf{u}}\left\Vert \mathbf{B}^{\dagger}\Gamma\right\Vert +\left\Vert \tilde{\mathbf{B}}\Gamma\right\Vert \right)\cdot\left(l_{\mathbf{w}}+l_{\mathbf{w}}^{\mathbf{x}}\left\Vert \mathbf{x}_{r}\right\Vert +l_{\mathbf{w}}^{\mathbf{x}}\left\Vert \mathbf{e}\right\Vert +l_{\mathbf{w}}^{\mathbf{u}}\left\Vert \hat{\mathbf{u}}\right\Vert +l_{\mathbf{w}}^{\mathbf{w}0}\left\Vert \mathbf{w}_{0}\right\Vert \right)
\end{array}\right)\nonumber \\
\le & -\bar{\mathbf{e}}^{\top}\left(2\mathbf{M}-\Delta{}_{t}^{\top}\mathbf{N}-\mathbf{N}\Delta_{t}-2\beta_{1}\sigma_{\max}(\mathbf{N})\mathbf{I}_{2n+k}\right)\bar{\mathbf{e}}\nonumber \\
 & +2\sigma_{\max}(\mathbf{N})\cdot\left(\begin{array}{c}
c_{\mathbf{x}r}\cdot\left(\beta_{1}+l_{\partial\mathbf{w}}^{\mathbf{x}}\left\Vert \mathbf{A}\right\Vert +l_{\partial\mathbf{w}}^{\mathbf{u}}\left\Vert \mathbf{B}^{\dagger}\mathbf{A}\right\Vert +\left\Vert \tilde{\mathbf{B}}\mathbf{A}\right\Vert \right)+c_{\dot{\mathbf{x}}r}\cdot\left(l_{\partial\mathbf{w}}^{\mathbf{u}}\left\Vert \mathbf{B}^{\dagger}\right\Vert +\left\Vert \tilde{\mathbf{B}}\right\Vert \right)\\
\begin{array}{c}
+c_{\mathbf{u}}\beta_{2}+c_{\mathbf{w}0}l_{\mathbf{w}}^{\mathbf{w}0}\beta_{0}+c_{\dot{\mathbf{w}}0}l_{\partial\mathbf{w}}^{\mathbf{w}_{0}}+(l_{\mathbf{v}}+c_{\mathbf{v}0}l_{\mathbf{v}}^{\mathbf{v}0})\left\Vert \mathbf{L}\Pi\right\Vert +l_{\mathbf{w}}\beta_{0}+l_{\partial\mathbf{w}}\end{array}
\end{array}\right)\cdot\left\Vert \bar{\mathbf{e}}\right\Vert ,\nonumber 
\end{align}

\noindent where the three scalars $\{\beta_{i}\}_{i=0,\thinspace1,\,2}$
are given as
\begin{equation}
\begin{aligned}\beta_{0} & :=l_{\partial\mathbf{w}}^{\mathbf{x}}\left\Vert \Gamma\right\Vert +l_{\partial\mathbf{w}}^{\mathbf{u}}\left\Vert \mathbf{B}^{\dagger}\Gamma\right\Vert +\left\Vert \tilde{\mathbf{B}}\Gamma\right\Vert ,\thinspace\thinspace\beta_{1}:=l_{\mathbf{v}}^{\mathbf{x}}\left\Vert \mathbf{L}\Pi\right\Vert +l_{\mathbf{w}}^{\mathbf{x}}\beta_{0},\thinspace\thinspace\beta_{2}:=l_{\mathbf{v}}^{\mathbf{u}}\left\Vert \mathbf{L}\Pi\right\Vert +l_{\partial\mathbf{w}}^{\mathbf{x}}\left\Vert \mathbf{B}\right\Vert +l_{\mathbf{w}}^{\mathbf{u}}\beta_{0}.\end{aligned}
\label{eq: beta0-3}
\end{equation}
\end{figure*}

\section*{References}

\bibliographystyle{model5-names}
\bibliography{HuCamachoXie_OutputFeedback_Automatica2016}

\end{document}